\documentclass[a4paper,11pt]{amsart}

\usepackage{amsmath,amssymb,amsfonts,latexsym}

\newtheorem*{theoA}{Theorem A}
\newtheorem*{theoB}{Theorem B}
\newtheorem*{theoC}{Theorem C}
\newtheorem*{theoD}{Theorem D}
\newtheorem*{theoE}{Theorem E}
\newtheorem*{theoF}{Theorem F}
\newtheorem*{theoG}{Theorem G}
\newtheorem*{theoH}{Theorem H}
\newtheorem*{theoI}{Theorem I}
\newtheorem*{theoJ}{Theorem J}
\newtheorem{theo}{Theorem}[section]

\newtheorem{lem}{Lemma}[section]

\theoremstyle{definition}
\newtheorem{exm}{Example}[section]
\newtheorem{ques}{Question}[section]
\newtheorem{defi}{Definition}[section]

\theoremstyle{remark}
\newtheorem{rem}{Remark}[section]

\newcommand{\ol}{\overline}
\newcommand{\be}{\begin{equation}}
\newcommand{\ee}{\end{equation}}
\newcommand{\beas}{\begin{eqnarray*}}
    \newcommand{\eeas}{\end{eqnarray*}}
\newcommand{\bea}{\begin{eqnarray}}
\newcommand{\eea}{\end{eqnarray}}

\numberwithin{equation}{section}
\renewcommand{\leq}{\leqslant}
\renewcommand{\geq}{\geqslant}

\begin{document}

\title[Uniqueness of meromorphic functions]{Uniqueness of $P(f)$ and $[P(f)]^{(k)}$ concerning weakly weighted sharing}

\author{Molla Basir Ahamed}
\address{Molla Basir Ahamed, School of Basic Science, Indian Institute of Technology Bhubaneswar, Bhubaneswar, 752050, Odisha, India}
\email{bsrhmd117@gmail.com}
\address{Molla Basir Ahamed, Kalipada Ghosh Tarai Mahavidyalaya, Bagdogra, West Bengal, 734014, India}
\email{bsrhmd2014@gmail.com}

\begin{abstract}
In this paper, with the help of the idea of weakly weighted sharing introduced by \emph{Lin -Lin} [Kodai Math. J., 29(2006), 269-280], we study the uniqueness of a polynomial expression $ P(f) $ and $ [P(f)]^{(k)} $ of a meromorphic function $ f $ sharing a small function. The main results significantly improved the result of \emph{Liu - Gu} [Kodai Math. J., 27(3)(2004), 272-279]. This research work explores certain condition under which the polynomial $ P(f) $ can be reduced to a non-zero monomial, and as a consequence, the specific form of the function $ f $ is obtained. By some constructive examples it has been shown that some conditions in the main results can not be removed and some of the inequalities are sharp.

\vspace{2mm}


\vspace{2mm}

\end{abstract}

\thanks{This work was supported by Institutional Post Doctoral Fellowship of Indian Institution of Technology Bhubaneswar.}


\maketitle


\section{Introduction, Definitions and Results}

Let $ \mathbb{C} $ be the complex plane, and let $ f $ be a non-constant meromorphic function defined on $ \mathbb{C} $. We assume that the reader is familiar with the standard definitions and notations used in the Nevanlinna value distribution theory, such as $ T(r,f),\; m(r,f),\ N(r,f) $ etc (see \cite{Hay & 1964, Yan & 1993, Yi & Yan & 1995}). By $ S(r,f) $ we denote any quantity the condition $ S(r,f)-o\left(T(r,f)\right) $ as $ r\rightarrow\infty $ possibly outside of an exceptional set $ E $ of  finite linear measure. A meromorphic function $ a\equiv a(z) $  is called a small function with respect to $ f $ if either $ a\equiv \infty $ or $ T(r,a)=S(r,f) $. throughout the paper, we denote $ S(f) $, the set of all small functions with respect to $ f $. One can easily verify that $ \mathbb{C}\cup\{\infty\}\subset S(f) $ and $ S(f) $ forms a field over the field of complex numbers.  \par For $ a\in\mathbb{C}\cup\{\infty\} $, the quantities $ \delta(a,f) $ and $ \Theta(a,f) $, defined as follows \beas \delta(a,f)=1-\limsup_{r\rightarrow\infty}\frac{N(r,a;f)}{T(r,f)} \eeas and \beas \Theta(a,f)=1-\limsup_{r\rightarrow\infty}\frac{\ol N(r,a;f)}{T(r,f)}, \eeas are respectively called the deficiency and ramification index of $ a $ for the function $ f $.\par For any two non-constant meromorphic functions $ f $ and $ g $, and $ a\in S(f) $, we say that $ f $ and $ g $ share $ a $ $ IM $ $ (CM) $ provided that $ f-a $ and $ g-a $ have the same set of zeros ignoring (counting) multiplicities. If $ 1/f $ and $ 1/g $ share $ 0 IM \; (CM) $, we say that $ f $ and $ g $ share $ \infty IM\; (CM). $
\par
\begin{defi}\cite{Lin & Lin & KMJ & 2006}
    Let $ N_E(r,a) $ be the counting function of all the common zeros of $ f-a $ and $ g-a $ with the same multiplicities, and $ N_0(r,a) $ be the counting function of all common zeros with ignoring multiplicities. We denote by $ \ol N_E(r,a) $ and $ \ol N_0(r,a) $ the reduced counting function of $ f $ and $ g $ corresponding to the counting functions $ N_E(r,a) $ and $ N_0(r,a) $, respectively. If \beas \ol N\left(r,\frac{1}{f-a}\right)+\ol N\left(r,\frac{1}{g-a}\right)-2\ol N_E(r,a)=S(r,f)+S(r,g), \eeas then we say that $ f $ and $ g $ share $ a $ $ CM $. On the other way, if \beas \ol N\left(r,\frac{1}{f-a}\right)+\ol N\left(r,\frac{1}{g-a}\right)-2\ol N_0(r,a)=S(r,f)+S(r,g), \eeas then we say that $ f $ and $ g $ share $ a $ $ IM $.
\end{defi}
\begin{defi}\cite{Lin & Lin & KMJ & 2006}
    Let $ k $ be a positive integer, and let $ f $ be a non-constant meromorphic function and $ a\in S(f) $.
    \begin{enumerate}
        \item[(a).] $ \displaystyle\ol N_{k)}\left(r,\frac{1}{f-a}\right) $ denotes the counting function of those $ a $-points of $ f $ whose multiplicities are not greater than $ k $, where each $ a $-point is counted only once.
        \item[(b).] $ \displaystyle\ol N_{(k}\left(r,\frac{1}{f-a}\right) $ denotes the counting function of those $ a $-points of $ f $ whose multiplicities are not less than $ k $, where each $ a $-point is counted only once.
        \item[(c).] $ \displaystyle\ol N_{k}\left(r,\frac{1}{f-a}\right) $ denotes the counting function of those $ a $-points of $ f $, where an $ a $-point of $ f $ with multiplicity $ m $ counted $ m $ times if $ m\leq k $ and $ k $ times if $ m>k $.
    \end{enumerate}
\end{defi}
\begin{defi}
    For $ a\in\mathbb{C}\cup\{\infty\} $ and $ p\in\mathbb{N} $, and for a meromorphic function $ f $, we denote by $ N_p\left(r,\frac{1}{f-a}\right) $ the sum \beas \ol N\left(r,\frac{1}{f-a}\right)+\ol N_{(2}\left(r,\frac{1}{f-a}\right)+\ldots+\ol N_{(p}\left(r,\frac{1}{f-a}\right). \eeas\par Clearly, $ \displaystyle N_1\left(r,\frac{1}{f-a}\right)= \ol N\left(r,\frac{1}{f-a}\right). $
\end{defi}
\begin{defi}
    We denote by $ \delta_k (a,f)$ the quantity \beas \delta_k(a,f)=1-\limsup_{r\rightarrow\infty}\frac{N_k\left(r,a;f\right)}{T(r,f)},
    \eeas where $ k $ is a positive integer. Clearly, $ \delta_k(a,f)\geq \delta(a,f). $
\end{defi}

\par From the last few decades, the uniqueness theory of entire or meromorphic functions has become a prominent branch of the value distribution theory (see \cite{Yi & Yan & 1995}). \emph{Rubel - Yang} \cite{Rub & Yan & SV & 1977} first established the result when an entire function $ f $ and its derivative $ f^{\prime} $ share two complex values $ a $ and $ b $ $ CM $, then they are identical i.e., $ f\equiv f^{\prime} $. An elementary calculation shows that the function will be of the form $ f(z)=ce^z $, where $ c $ is a non-zero constant. In $ 1979 $,  improving the result in \cite{Lah & Dew & KMJ & 2003}, analogous result corresponding to IM sharing was obtained by \textit{Mues - Steinmetz} \cite{Mue & Ste & CVTA & 1986}.\par In course of time, many researchers such as \emph{Br$ \ddot{u} $ck} \cite{Bru & RM & 1996}, \emph{Ahamed} \cite{Aha & CKMS & 2018}, \emph{Banerjee - Ahamed} \cite{Ban & Aha & 2015 & AUPOFRN,Ban & Aha & AUVTSMI & 2016,Aha & Ban & BTUB & 2017,Ban & Aha & EJMAA & 2018}, \emph{Gundersen} \cite{Gun & JMMA & 1980}, \emph{Yang} \cite{Yan & BAMS & 1990} et al. became more involved to find out the relation between an entire or meromorphic function with its higher order derivatives or with some general (linear) differential polynomials, sharing one value or sets of values. Finding the class of the functions, \textit{Yang - Zhang} \cite{Yan & Zha &AM & 2008} (see also \cite{Zha & KMJ & 2009}) first considered the uniqueness of a power of a meromorphic (entire)
function $ F=f^n $ and its derivative $ F^{\prime} $ when they share certain value. \par In the paper,
\textit{Yang-Zhang} \cite{Yan & Zha &AM & 2008} explores the class of the functions satisfying some differential equations of some special forms. Now we are invoking the following results which elaborates the gradual developments to this setting of meromorphic functions. \textit{Zhang} \cite{Zha & KMJ & 2009} proved a theorem, which improved all the results obtained in \cite{Yan & Zha &AM & 2008}.
\par In $ 2003 $, \emph{Yu} \cite{Yu & JIPAM & 2003} considered the uniqueness problem of entire and meromorphic functions when it shares one small functions with its derivative, and proved the following results.

\begin{theoA}\cite{Yu & JIPAM & 2003}
    Let $ k $ be a positive integer, and $ f $ be a non-constant entire function and $ a\in S(f) $ and $ a\not\equiv 0, \infty $. If $ f $ and $ f^{(k)} $ share $ a $ $ CM $ and $ \delta(0;f)>\frac{3}{4} $, then $ f\equiv f^{(k)} $.
\end{theoA}
\begin{theoB}\cite{Yu & JIPAM & 2003}
    Let $ k $ be a positive positive integer, and $ f $ be a non-constant meromorphic function, $ a\in S(f) $ and $ a\not\equiv 0, \infty $, $ f $ and $ a $ do not have any common pole. If $ f $ and $ f^{(k)} $ share $ a $ $ CM $ and \beas 4\delta(0,f)+2(k+8)\Theta(\infty,f)>2k+19, \eeas then $ f\equiv f^{(k)} $
\end{theoB}
\par In the same paper, \emph{Yu} \cite{Yu & JIPAM & 2003} posed the following open questions on which many researchers investigated and later established results by answering them.
\begin{enumerate}
    \item[(i).] Can a $ CM $ shared value be replaced by an $ IM $ shared value in \emph{Theorem A} ?
    \item[(ii).] Is the condition $ \delta(0,f)>\frac{3}{4} $ sharp in \emph{Theorem A} ?
    \item[(iii).] Is the condition $ 4\delta(0,f)+2(k+8)\Theta(\infty,f)>2k+19 $ sharp in \emph{Theorem B} ?
    \item[(iv).] Can the condition, ``$ f $ and $ a $ do not have any common pole", be deleted in \emph{Theorem B} ?
\end{enumerate}\par In $ 2004 $, \emph{Liu - Gu} \cite{Liu & Gu & KMJ & 2004} applied a different method of proof, and obtained the following results.
\begin{theoC}\cite{Liu & Gu & KMJ & 2004}
    Let $ f $ be a non-constant meromorphic function, $ a\in S(f) $ and $ a\not\equiv 0, \infty $. If $ f $ and $ f^{(k)} $ share the value $ a $ $ CM $, and $ f $ and $ a $ do not have any common pole of same multiplicity and $ 2\delta(0,f)+4\Theta(\infty,f)>5 $, then $ f\equiv f^{(k)} $.
\end{theoC}
\begin{theoD}\cite{Liu & Gu & KMJ & 2004}
    Let $ f $ be a non-constant entire function, $ a\in S(f) $ and $ a\not\equiv 0, \infty $. If $ f $ and $ f^{(k)} $ share the value $ a $ $ CM $, and $ \delta(0,f)>\frac{1}{2} $, then $ f\equiv f^{(k)} $.
\end{theoD}
\par In $ 2006 $, \emph{Lin - Lin} \cite{Lin & Lin & KMJ & 2006} introduced the following notion of weakly weighted sharing of values which is a scaling between $ CM $ and $ IM $ sharing.  Let $ f $ and $ g $ be two non-constant meromorphic functions sharing $ a $ $ IM $. For $ a\in S(f)\cap S(g) $ and a positive integer $ k $ or $ \infty $,
\begin{enumerate}
    \item[(i).] $ \ol N^{E}_{k)}(r,a) $ denotes the counting function of those $ a $-points of $ f $ whose multiplicities are equal to the corresponding   $ a $-points of $ g $, both of their multiplicities are not greater than $ k $, where each $ a $-point is counted only once.
    \item[(ii).] $ \ol N^{0}_{(k}(r,a) $ denotes the reduced counting function of those $ a $-points of $ f $ which are $ a $-points og $ g $, both of their multiplicities are not less than $ k $, where each $ a $-point is counted only once.
\end{enumerate}
\begin{defi}\cite{Lin & Lin & KMJ & 2006}
    For $ a\in S(f)\cap S(g) $, if $ k $ be a positive integer or $ \infty $, and \beas && \ol N_{k)}(r,a;f)+\ol N_{k)}(r,a;g)-2\ol N^{E}_{k)}(r,a)=S(r,f)+S(r,g)\\ && \ol N_{(k+1}(r,a;f)+\ol N_{(+1}(r,a;g)-2\ol N^{0}_{(k+1}(r,a)=S(r,f)+S(r,g) \eeas or, if $ k=0 $ and \beas \ol N(r,a;f)+\ol N(r,a;g)-2\ol N_0(r,a)=S(r,f)+S(r,g), \eeas then we say that $ f $ and $ g $ weakly share $ a $ with weight $ k $. Here, we write, $ f $ and $ g $ share $ (a,k) $ to mean that $ f $ and $ g $ share the value $ a $ weakly with weight $ k $. \par Obviously if $ f $ and $ g $ share $ (a, k) $, then $ f $ and $ g $ share $ (a,p) $ for any $ p\; (0\leq p\leq k) $. Also, we note that $ f $ and $ g $ share $ a $ $ IM $ or $ CM $ if and only
    if $ f $ and $ g $ share $ (a,0) $ or $ (a,\infty) $, respectively.
    Suppose $ \mathcal{F} $ and $ \mathcal{G} $ share $ 1 $ $ IM $. By  $N_L(r,1; \mathcal{F})$ we denotes the counting function of the $ 1 $-points of $ \mathcal{F} $ whose multiplicities are greater than $ 1 $-points of $ \mathcal{G} $, $N_L(r,1; \mathcal{G})$ is defined similarly.
\end{defi}
\par With the help of the notion of weakly weighted sharing, \emph{Lin - Lin} \cite{Lin & Lin & KMJ & 2006} investigated the uniqueness problem between a meromorphic function $ f $ and its $ k $th derivative  sharing a small function, and proved  results as follows.
\begin{theoE}\cite{Lin & Lin & KMJ & 2006}
    Let $ k\geq 1 $ and $ 2\leq m\leq \infty $. Let $ f $ be a non-constant meromorphic function, $ a\in S(f) $ and $ a\not\equiv 0, \infty $. If $ f-a $ and $ f^{(k)}-a $ share $ (0,m) $  and \beas 2\delta_{k+2}(0,f)+4\Theta(\infty,f)>5, \eeas then $ f\equiv f^{(k)} $.
\end{theoE}
\begin{theoF}\cite{Lin & Lin & KMJ & 2006}
    Let $ k\geq 1 $, and  $ f $ be a non-constant meromorphic function, $ a\in S(f) $ and $ a\not\equiv 0, \infty $. If $ f-a $ and $ f^{(k)}-a $ share $ (0,1) $  and \beas 5\delta_{k+2}(0,f)+(k+9)\Theta(\infty,f)>k+12, \eeas then $ f\equiv f^{(k)} $.
\end{theoF}
\begin{theoG}\cite{Lin & Lin & KMJ & 2006}
    Let $ k\geq 1 $, and  $ f $ be a non-constant meromorphic function, $ a\in S(f) $ and $ a\not\equiv 0, \infty $. If $ f-a $ and $ f^{(k)}-a $ share $ (0,0) $  and \beas 5\delta_{k+2}(0,f)+(2k+7)\Theta(\infty,f)>2k+11, \eeas then $ f\equiv f^{(k)} $.
\end{theoG}
\begin{rem}
    To obtain the identical relation $ f\equiv f^{(k)} $ as a conclusion in \emph{Theorems E, F, G}, the respective conditions $ 2\delta_{k+2}(0,f)+4\Theta(\infty,f)>5 $, $ 5\delta_{k+2}(0,f)+(k+9)\Theta(\infty,f)>k+12 $ and $ 5\delta_{k+2}(0,f)+(2k+7)\Theta(\infty,f)>2k+11, $ can not be removed. The following example, ensures this fact.
\end{rem}
\begin{exm}
    Let $ f(z)=\displaystyle\frac{ze^z}{e^z+1}. $ Then we see that \beas f(z)-1=\frac{z-e^{-z}-1}{e^{-z}+1}\;\; \text{and}\;\; f^{\prime}(z)-1=-\frac{e^{-z}\left(z-e^{-z}-1\right)}{\left(e^{-z}+1\right)^2}.  \eeas Then $ f $ and $ f^{\prime} $ share the value $ 1 $ $ CM $, and $ \Theta(\infty,f)=0 $, $ \delta_{p}(0,f)=1 $ for $ p\geq 2 $. Then we see that \begin{enumerate}
        \item[(i).] $ 2\delta_{k+2}(0,f)+4\Theta(\infty,f)=2\ngtr 5. $
        \item[(ii).] $ 5\delta_{k+2}(0,f)+(k+9)\Theta(\infty,f)=5\ngtr k+12 $.
        \item[(iii).] $ 5\delta_{k+2}(0,f)+(2k+7)\Theta(\infty,f)=5\ngtr 2k+11 $,
    \end{enumerate}
    and we see that  $ f\not\equiv f^{\prime}. $
\end{exm}
\par Later, in $ 2011 $, \emph{Xu - Hu} \cite{Xu & Hu & GM & 2011} generalized \emph{Theorems E, F} and \emph{G} by considering \beas L(f)=a_kf^{(k)}+a_{k-1}f^{(k-1)}+\ldots+a_0f \eeas and  proving the following results.
\begin{theoH}\cite{Xu & Hu & GM & 2011}
    Let $ k\geq 1 $ and $ 2\leq m\leq \infty $. Let $ f $ be a non-constant meromorphic function, $ a\in S(f) $ and $ a\not\equiv 0, \infty $. If $ f-a $ and $ L(f)-a $ share $ (0,m) $  and \beas 2\delta_{k+2}(0,f)+4\Theta(\infty,f)>5, \eeas then $ f\equiv L(f) $.
\end{theoH}
\begin{theoI}\cite{Xu & Hu & GM & 2011}
    Let $ k\geq 1 $, and  $ f $ be a non-constant meromorphic function, $ a\in S(f) $ and $ a\not\equiv 0, \infty $. If $ f-a $ and $ L(f)-a $ share $ (0,1) $  and \beas \delta_{k+2}(0,f)+\frac{3}{2}\delta_{2}(0,f)+\left(\frac{7}{2}+k\right)\Theta(\infty,f)>k+5, \eeas then $ f\equiv L(f) $.
\end{theoI}
\begin{theoJ}\cite{Xu & Hu & GM & 2011}
    Let $ k\geq 1 $, and  $ f $ be a non-constant meromorphic function, $ a\in S(f) $ and $ a\not\equiv 0, \infty $. If $ f-a $ and $ L(f)-a $ share $ (0,0) $  and \beas 2\delta_{k+2}(0,f)+\delta_{2}(0,f)+2\Theta(0,f)+(2k+6)\Theta(\infty,f)>2k+10, \eeas then $ f\equiv L(f) $.
\end{theoJ}
\begin{rem}
    To get the relation $ f\equiv L(f) $ in \emph{Theorems H, I, J}, the respective conditions $ 2\delta_{k+2}(0,f)+4\Theta(\infty,f)>5 $, $ \delta_{k+2}(0,f)+\displaystyle\frac{3}{2}\delta_{2}(0,f)+\displaystyle\left(\frac{7}{2}+k\right)\Theta(\infty,f)>k+5 $ and $ 2\delta_{k+2}(0,f)+\delta_{2}(0,f)+2\Theta(0,f)+(2k+6)\Theta(\infty,f)>2k+10, $ can not be removed. From the following example, one ensures this.
\end{rem}
\begin{exm}
    Let $ f(z)=\displaystyle\frac{e^{z}}{e^{e^z}-1} $ and \beas  L(f)=\displaystyle\frac{2e^z-1}{e^{2z}}f^{\prime}(z)+\frac{2e^{2z}-2e^z+1}{e^{2z}}f(z).\eeas Then we see that \beas f(z)-e^z=\frac{2e^z-e^{e^z}}{e^{e^z}-1}\;\; \text{and}\;\; L(f)-e^z=\frac{2e^z-e^{e^z}}{\left(e^{e^z}-1\right)^2}.  \eeas Then $ f $ and $ L(f) $ share the value $ a(z)=e^z $\; $ CM $, and $ \Theta(\infty,f)=0 $, $ \delta_{p}(0,f)=1 $ for $ p\geq 2 $. Then we see that \begin{enumerate}
        \item[(i).] $  2\delta_{k+2}(0,f)+4\Theta(\infty,f)=2\ngtr 5. $
        \item[(ii).] $ \delta_{k+2}(0,f)+\displaystyle\frac{3}{2}\delta_{2}(0,f)+\left(\displaystyle\frac{7}{2}+k\right)\Theta(\infty,f)=1+\displaystyle\frac{3}{2}\ngtr k+5 $.
        \item[(iii).] $ 2\delta_{k+2}(0,f)+\delta_{2}(0,f)+2\Theta(0,f)+(2k+6)\Theta(\infty,f)=5\ngtr 2k+10 $,
    \end{enumerate}
    and we see that  $ f\not\equiv L(f). $
\end{exm}
\begin{rem}\label{rem1.3}
Regarding finding the class of the meromorphic functions which satisfies $ f\equiv f^{(k)} $ or $ f\equiv L(f) $, we have the following observations.
\begin{enumerate}
    \item[(i)] When a non-constant meromorphic function $ f $ satisfies the equation $ f\equiv f^{(k)} $ or $ f\equiv L(f) $, then obviously the function $ f $ can not have any pole but may have zeros. Therefore, in this case, the solution  function $ f $ must be an entire.
    \item[(ii).] The general solutions of the differential equation $ f\equiv f^{(k)} $, is \beas f(z)=c_1e^{z}+c_2e^{\theta ^2z}+\ldots+c_2e^{\theta ^{k-1}z}, \eeas where all the $ c_i $ $ (i=1, 2, \ldots, k) $ are arbitrary complex constants, and $ \theta=\cos\left(\frac{2\pi}{k}\right)+i\sin\left(\frac{2\pi}{k}\right).$
    \item[(iii).]  On the other hand, the general solution of the differential equation $ f=L(f) $ is in general \beas f(z)=b_1 e^{\beta_1 z}+b_2 e^{\beta_2 z}+\ldots+b_k e^{\beta_k z}, \eeas where all the $ d_i $ $ (i=1, 2, \ldots, k) $ are arbitrary complex constants, and $ \beta_{j}\; (j=1, 2, \ldots, k) $ are the roots of the equation \beas w^k+a_{k-1}w^{k-1}+\ldots+a_1w+a_0=0. \eeas
    \item[(iv).] For any positive integer $ n $, if we choose $ f(z)=c e^{\frac{\mu}{n}z} $, where $ \mu $ is a complex constant satisfying $ \mu^k=1 $, then it is not hard to verify that $ f^n(z)-a(z) $ and $ \left(f^n(z)\right)^{(k)}-a(z) $  share $ 0 $ $ CM $, where $ a(z) $ is a small function with respect to $ f $, and above all we have $ f^n(z)=\left(f^n(z)\right)^{(k)}. $
\end{enumerate}

\end{rem}

\par It is therefore natural to investigate on the question : \emph{what happen if we consider some power of a meromorphic function $ f $ so that $ f^n-a(z) $ and $ \left(f^n\right)^{(k)}-a(z) $ share $ 0 $ $ CM $ ?} \par In this particular direction, \emph{Zhang} \cite{Zha & KMJ & 2009} and \emph{Zhang - Yang} \cite{Zha & Yan & AASFM & 2007} answered the above question and obtained a uniqueness result between $ f^n $ and $ \left(f^n\right)^{(k)} $, and shown that, the function $ f $ actually takes the form \beas f(z)=c e^{\frac{\lambda}{n}z},  \eeas where $ c $ is a non-zero complex number and $ \lambda^k=1 $.
\par
We know that a linear differential equation with constant coefficients can be solved and the general solution can be obtained as the linear combination of the independent solutions of the that equations. In the above theorems, the researchers hence found the solution functions. But what could be the possible relationship and hence the solutions class if we considering polynomial expression of a meromorphic function $ f $ and its derivatives in the \emph{Theorem E} to \emph{Theorem J}, no attempts till now made by any researcher. Investigating and exploring  the above situation is the main motivation of writing this paper.\par Henceforth, throughout this paper, for a meromorphic function $ f $, we consider a polynomial expression $ P(f) $, which is a more general setting of power of $ f $, defined as
\beas P(f)=a_nf^n+\ldots+a_1f+a_0. \eeas
In connection with the above discussions, it is therefore reasonable to raise some questions as below.
\begin{ques}
    \begin{enumerate}
        \item[(i).]     What happens if $ P(f)-a $ and $ [P(f)]^{(k)}-a $ share $ (0,m) $ in all the above mentioned results ?
        \item[(ii).]     Can we get an uniqueness relation between $ P(f)$ and $ [P(f)]^{(k)}$ ?\item[(iii).]     Can we also get a specific form of the function $ f $ if the answer of the questions is true ?
    \end{enumerate}

\end{ques}
In this paper, taking the above questions into background, we investigate to find the possible answers of them. To make our investigation easier, we will use some transformation. Hence we factorize the expression $ P(f) $ as \beas P(f)=a_n\left(z-d_{p_1}\right)^{p_1}\left(z-d_{p_2}\right)^{p_2}\ldots\left(z-d_{p_s}\right)^{p_s}, \eeas where $ a_j\; (j=0, 1, 2, \ldots, n-1), $ $ a_n(\neq 0) $ and $ d_{p_i}\; (i=1, 2, \ldots, s) $ are distinct finite complex numbers, and $ p_1, p_2, \ldots, p_s $, $ n $ and $ k $ all are positive integers with $\displaystyle \sum_{j=1}^{s}p_j=n $.  Let $ p=\max\{p_1, p_2, \ldots, p_s\} $, and we consider an arbitrary polynomial \beas  Q(f_*)&=&a_n\prod_{j=1, p_j\neq p}^{s}\left(f_{*}+d_p-d_{p_j}\right)^{p_j}\\ &=& c_mf_{*}^m+\ldots+c_1f_*+c_0, \eeas where $ a_n=c_m $, $ f_*=f-d_p $ and $ m=n-p $. Obviously, we have \beas P(f)=f_*^pQ(f_*). \eeas In particular, when $ d_p=0 $, then it is not hard to get $ f_*=f $ and $ P(f)=f^pQ(f) $. \par We now state the main results of this paper as follows.
\begin{theo}\label{th1.1}
    Let $ k\geq 1 $ and $ 2\leq m\leq \infty $, be two integers. Let $ f $ be a non-constant meromorphic function, and $ a\in S(f) $ with $ a\not\equiv 0, \infty $. If $ P(f)-a $ and $ [P(f)]^{(k)}-a $ share $ (0,m) $ and $ 2n\delta_{k+2}(0,f)+4\Theta(\infty,f)>5 $, then $ P(f)\equiv [P(f)]^{(k)} $ i.e., $ f_{*}^pQ(f_*)=[ f_{*}^pQ(f_*)]^{(k)} $. \par  Furthermore, if $ p>k+1 $ then
    \begin{enumerate}
        \item[(i).] $ Q(f_*) $ reduces to a non-zero monomial $ c_jf_{*}^j (\not\equiv 0) $ for some $ j\in \{0,1,2, \ldots,m\} $.
        \item[(ii).]  $ f(z) $ takes the form \beas f(z)=c e^{\frac{\lambda}{p+j}z}+d_p, \eeas where $ c $ is a non-zero constant and $ \lambda^k=1 $.
    \end{enumerate}
\end{theo}

\begin{theo}\label{th1.2}
    Let $ k(\geq 1) $ be an integer, $ f $ be a non-constant meromorphic function, and $ a\in S(f) $ where $ a\not\equiv 0, \infty $. If $ P(f)-a $ and $ [P(f)]^{(k)}-a $ share $ (0,1) $ and \beas  {5n}\delta_{k+2}(0,P(f))+(k+9)\Theta(\infty,f)>{3n+k+9},\eeas  then the conclusion of \emph{Theorem \ref{th1.1}} holds.
\end{theo}
\begin{rem}
    Next example ensures that conclusion of \emph{Theorem \ref{th1.2}} ceases to be hold if we remove the condition \beas {5n}\delta_{k+2}(0,P(f))+(k+9)\Theta(\infty,f)>{3n+k+9}. \eeas
\end{rem}
\begin{exm}
    Let \beas f(z)=-\frac{a_1}{2a_2}+\frac{1}{18a_2}\left(9e^{3z}+6z+2\right)^{1/2}, \eeas where $ a_1, a_2 $ are two non-zero constants and $ P(f) =a_2f^2(z)+a_1f(z)$. Then we see that $ P(f)=\displaystyle\frac{1}{9}({9e^{3z}+6z+2}) $, and also \beas 3(P(f)-z)=[P(f)]^{\prime}-z. \eeas Thus, clearly $ P(f)-z $ and $ [P(f)]^{\prime}-z $ share $ (0,\infty), $ and $ \delta_{3}(0,P(f))=0 $ and $ \Theta(\infty,f)=1. $ Clearly \beas {5n}\delta_{k+2}(0,P(f))+(k+9)\Theta(\infty,f)=10\ngtr 17= 3n+k+9. \eeas Hence $ P(f)\not\equiv [P(f)]^{(k)}. $
\end{exm}
\begin{theo}\label{th1.3}
    Let $ k(\geq 1) $ be an integer, $ f $ be a non-constant meromorphic function, and $ a\in S(f) $ where $ a\not\equiv 0, \infty $. If $ P(f)-a $ and $ [P(f)]^{(k)}-a $ share $ 0 $ $ IM $ and \beas  (k-1)\Theta(\infty,f)+\delta_{2}(\infty,f)+n[\delta_k(0,P(f))+\delta_{k+1}(0,P(f))]>k+n ,\eeas then the conclusion of \emph{Theorem \ref{th1.1}} holds.
\end{theo}
\begin{rem}
    The condition \beas (k-1)\Theta(\infty,f)+\delta_{2}(\infty,f)+n[\delta_k(0,P(f))+\delta_{k+1}(0,P(f))]>k+n \eeas in \emph{Theorem \ref{th1.3}}\; is sharp which can be seen from the next example.
\end{rem}
\begin{exm}
    Suppose that, $ P(f)=f $, and $ k=1 $, where \beas f(z)=\frac{2e^{2z}}{e^{2z}-1}. \eeas  \par We see that \beas N(r,f)\sim T(r,f),\;\; \text{and so we have }\;\; \delta_{2}(\infty,f)=0.  \eeas\par We also see that $ \Theta(0,P(f))=\delta_{2}(0,P(f))=1. $ \par Then $ P(f) $ and $ [P(f)]^{(k)} $ share the value $ 1 $ $ IM $, and \beas (k-1)\Theta(\infty,f)+\delta_{2}(\infty,f)+n[\delta_k(0,P(f))+\delta_{k+1}(0,P(f))]=k+n \eeas but $ P(f)\not\equiv [P(f)]^{(k)}. $
\end{exm}
\section{\sc Some lemmas}
In this section, we are going to present some lemmas which will needed to prove our main results.\par We define the functions $ \mathfrak{F} $, $ \mathfrak{G} $ and $ \mathfrak{H} $ as follows,
\bea\label{e2.1} \mathfrak{F}=\frac{P(f)}{a},\;\;\; \mathfrak{G}=\frac{[P(f)]^{(k)}}{a}. \eea \bea\label{e2.2} \mathfrak{H}=\left(\frac{\mathfrak{F}^{\prime\prime}}{\mathfrak{F}^{\prime}}-2\frac{\mathfrak{F}^{\prime}}{\mathfrak{F}-1}\right)-\left(\frac{\mathfrak{G}^{\prime\prime}}{\mathfrak{G}^{\prime}}-2\frac{\mathfrak{G}^{\prime}}{\mathfrak{G}-1}\right). \eea

\begin{lem}\cite{Yan & MZ & 1972}\label{lem2.1}
    Let $ f $ be a non-constant meromorphic function and let \beas P(f)=a_nf^n+a_{n-1}f^{n-1}+\ldots+a_1f+a_0, \eeas where $ a_i\in S(f) $ for $ i=0,1,2 \ldots, n; $ $ a_n\neq 0 $, be a polynomial in $ f $ of degree $ n $. Then \beas T\left(r, P(f)\right)= n\; T(r,f)+S(r,f). \eeas
\end{lem}
\begin{lem}\cite{Lah & Dew & KMJ & 2003}\label{lem22.2}
    If $ N\left(r,\displaystyle\frac{1}{g^{(k)}}\bigg|\; g\neq 0 \right) $ denotes the counting function of those zeros of $ g^{(k)} $ which are not the zeros of $ g $, where a zero of $ g^{(k)} $ is counted according to its multiplicity, then \beas N\left(r,\frac{1}{g^{(k)}}\bigg| g \neq 0 \right)\leq k \ol N(r,g)+N_{k)}\left(r,\frac{1}{g}\right)+k\ol N_{(k}\left(r,\frac{1}{g}\right)+S(r,g) \eeas
\end{lem}
\begin{lem}\cite{Yi & Yan & 1995}\label{lem2.2}
    Let $ g $ be a non-constant meromorphic function, and let $ k $ be a positive integer. Then \begin{enumerate}
        \item[(i).]  $ \displaystyle N\left(r,\frac{1}{g^{(k)}}\right)\leq N\left(r,\frac{1}{g}\right)+k\ol N(r,g)+S(r,g). $
        \item[(ii).]  $ \displaystyle N\left(r,\frac{1}{g^{(k)}}\right)\leq T\left(r,g^{(k)}\right)-T(r,g)+N\left(r,\frac{1}{g}\right)+S(r,g). $
    \end{enumerate}
\end{lem}
\begin{lem}\label{lem2.3}
    Let $ f $ be a non-constant meromorphic function, and let $ k $ be a positive integer. Then \begin{enumerate}
        \item[(i).]  $ \displaystyle N_2\left(r,\frac{1}{[P(f)]^{(k)}}\right)\leq N\left(r,\frac{1}{P(f)}\right)+k\ol N(r,f)+S(r,f). $
        \item[(ii).]  $ \displaystyle N_2\left(r,\frac{1}{[P(f)]^{(k)}}\right)\leq T\left(r,P[f]^{(k)}\right)-nT(r,f)+N_{k+2}\left(r,\frac{1}{P(f)}\right)+S(r,f). $
    \end{enumerate}
\end{lem}
\begin{proof}
    (i). By (i) of \emph{Lemma \ref{lem2.2}}, replacing $ g $ by $ P(f) $, wee see that \beas && N_2\left(r,\frac{1}{[P(f)]^{(k)}}\right)+\sum_{j=3}^{\infty}\ol N\left(r,\frac{1}{[P(f)]^{(k)}}\bigg| \geq j\right) \\&\leq & N_{k+2}\left(r,\frac{1}{P(f)}\right)+\sum_{j=k+3}^{\infty}\ol N\left(r,\frac{1}{P(f)}\bigg| \geq j\right)+k\ol N(r,P(f))+S(r,f). \eeas i.e., \beas && N_2\left(r,\frac{1}{[P(f)]^{(k)}}\right)\\ &\leq& N_{k+2}\left(r,\frac{1}{P(f)}\right)+\sum_{j=k+3}^{\infty}\ol N\left(r,\frac{1}{P(f)}\bigg| \geq j\right)-\sum_{j=3}^{\infty}\ol N\left(r,\frac{1}{[P(f)]^{(k)}}\bigg| \geq j\right)\\&&+k\ol N(r,P(f))+S(r,f)\\&\leq& N_{k+2}\left(r,\frac{1}{P(f)}\right)+k\ol N(r,f)+S(r,f) \eeas
    (ii). We have \beas && N_2\left(r,\frac{1}{[P(f)]^{(k)}}\right)\\  &\leq& N\left(r,\frac{1}{[P(f)]^{(k)}}\right)-\sum_{p=3}^{\infty}N\left(r,\frac{1}{[P(f)]^{(k)}}\bigg|\geq p \right)\\&=& T\left(r,[P(f)]^{(k)}\right)-m\left(r,\frac{1}{[P(f)]^{(k)}}\right)-\sum_{p=3}^{\infty}N\left(r,\frac{1}{[P(f)]^{(k)}}\bigg|\geq p \right)+O(1)\\&\leq& T\left(r,[P(f)]^{(k)}\right)-m\left(r,\frac{1}{P(f)}\right)-m\left(r,\frac{[P(f)]^{(k)}}{P(f)}\right)-\sum_{p=3}^{\infty}N\left(r,\frac{1}{[P(f)]^{(k)}}\bigg|\geq p \right)\\&&+S(r,f)\\&\leq& T\left(r,[P(f)]^{(k)}\right)-nT\left(r,f\right)+N\left(r,\frac{1}{P(f)}\right)\\&&-\sum_{p=3}^{\infty}N\left(r,\frac{1}{[P(f)]^{(k)}}\bigg|\geq p \right)+S(r,f)\\&\leq& T\left(r,[P(f)]^{(k)}\right)-nT\left(r,f\right)+N_{k+2}\left(r,\frac{1}{P(f)}\right)+\sum_{p=k+3}^{\infty}\ol N\left(r,\frac{1}{P(f)}\bigg |\geq p\right)\\&&-\sum_{p=3}^{\infty}N\left(r,\frac{1}{[P(f)]^{(k)}}\bigg|\geq p \right)+S(r,f)\\&\leq& T\left(r,[P(f)]^{(k)}\right)-nT\left(r,f\right)+N_{k+2}\left(r,\frac{1}{P(f)}\right)+S(r,f) \eeas
\end{proof}
\begin{lem}\cite{Xu & Hu & GM & 2011}\label{lem2.4}
    Let $ \mathfrak{F} $ and $ \mathfrak{G} $ be two non-constant meromorphic functions such that they share (1,0), then \beas \ol N_L\left(r,\frac{1}{\mathfrak{F}-1}\right)\leq \frac{1}{2}\ol N\left(r,\frac{1}{\mathfrak{F}}\right)+\frac{1}{2}\ol N(r,\mathfrak{F})+S(r,\mathfrak{F}). \eeas
\end{lem}
\begin{lem}\cite{Xu & Hu & GM & 2011}\label{lem2.5}
    Let $ \mathfrak{F} $ and $ \mathfrak{G} $ be two non-constant meromorphic functions such that they share (0,0), then \beas \ol N_L\left(r,\frac{1}{\mathfrak{F}-1}\right)\leq \ol N\left(r,\frac{1}{\mathfrak{F}}\right)+\ol N(r,\mathfrak{F})+S(r,\mathfrak{F}). \eeas
\end{lem}
\begin{lem}\cite{Lin & Lin & KMJ & 2006}\label{lem2.6}
    Let $ m $ be a non-negative integer or $ \infty $. Let $ \mathfrak{F} $ and $ \mathfrak{G} $ be two non-constant meromorphic functions sharing $ (1,m) $ and $ \mathfrak{H} $ be given by (\ref{e2.2}). If $ \mathfrak{H}\not\equiv 0 $, then
    \begin{enumerate}
        \item[(i).] for $ 2\leq m\leq \infty $ \beas && T(r,\mathfrak{F})\\ &\leq& N_2(r,\mathfrak{F})+N_2\left(r,\frac{1}{\mathfrak{F}}\right)+N_2(r,\mathfrak{G})+N_2\left(r,\frac{1}{\mathfrak{G}}\right)+S(r,\mathfrak{F})+S(r,\mathfrak{G}). \eeas
        \item[(ii).] for $ m=1 $ \beas && T(r,\mathfrak{F})\\ &\leq& N_2(r,\mathfrak{F})+N_2\left(r,\frac{1}{\mathfrak{F}}\right)+N_2(r,\mathfrak{G})+N_2\left(r,\frac{1}{\mathfrak{G}}\right)+\ol N^{L}\left(r,\frac{1}{\mathfrak{F}-1}\right)\\&&+S(r,\mathfrak{F})+S(r,\mathfrak{G}). \eeas
        \item[(iii).] for $ m=0 $ \beas && T(r,\mathfrak{F})\\ &\leq& N_2(r,\mathfrak{F})+N_2\left(r,\frac{1}{\mathfrak{F}}\right)+N_2(r,\mathfrak{G})+N_2\left(r,\frac{1}{\mathfrak{G}}\right)+2\ol N^{L}\left(r,\frac{1}{\mathfrak{F}-1}\right)\\&&+\ol N^{L}\left(r,\frac{1}{\mathfrak{G}-1}\right)+S(r,\mathfrak{F})+S(r,\mathfrak{G}). \eeas
    \end{enumerate}
    The same inequality holds also for $ T(r,\mathfrak{G}). $
\end{lem}
\begin{lem}\label{lem2.7}
    Let $ f $ be a transcendental meromorphic function, and $ \alpha(\not\equiv 0, \infty) $ be a meromorphic function such that $ T(r,\alpha)=S(r,f) $. Let $ b, c $ are any two finite non-zero distinct complex number. If \bea\label{e2.3} \Psi(f)=\alpha P(f)[P(f)]^{(k)}, \eea  where $ k\geq 1 $ is an integer, then \beas && 2nT\left(r,f\right)\\&\leq& 2 N\left(r,\frac{1}{ P(f)}\right)+N\left(\frac{1}{\Psi(f)-b}\right)+N\left(r,\frac{1}{\Psi(f)-c}\right)-N(r,\Psi(f))\\&&-N\left(r,\displaystyle\frac{1}{\Psi^{\prime}(f)}\right)+S(r,f). \eeas
\end{lem}
\begin{proof}
    Since $ f $ is a non-constant meromorphic function, hence one can check that the function $ \Psi(f) $ is also non-constant. We get from (\ref{e2.3}) that \beas \frac{1}{\alpha [P(f)]^2}=\frac{1}{\Psi(f)}\frac{[P(f)]^{(k)}}{P(f)}. \eeas Thus we have \beas m\left(r,\frac{1}{\alpha [P(f)]^2}\right)\leq m\left(r,\frac{1}{\Psi(f)}\right)+m\left(r,\frac{[P(f)]^{(k)}}{P(f)}\right)+O(1), \eeas \beas m\left(r,\frac{1}{\alpha [P(f)]^2}\right)=T\left(r,\alpha [P(f)]^2\right)-N\left(r,\frac{1}{\alpha [P(f)]^2}\right)+O(1) \eeas and \beas m\left(r,\frac{1}{\Psi(f)}\right)=T\left(r,\frac{1}{\Psi(f)}\right)-N\left(r,\frac{1}{\Psi(f)}\right)+O(1). \eeas\par Combining all the above relations, we can obtain as \bea\label{e2.4} && T\left(r,\alpha [P(f)]^2 \right)\\&\leq&\nonumber N\left(r,\frac{1}{\alpha [P(f)]^2}\right)+T(r,\Psi(f))-N\left(r,\frac{1}{\Psi(f)}\right)+m\left(r,\frac{[P(f)]^{(k)}}{P(f)}\right)\\&&+O(1)\nonumber\\&\leq& N\left(r,\frac{1}{\alpha [P(f)]^2}\right)+T(r,\Psi(f))-N\left(r,\frac{1}{\Psi(f)}\right)+O(1).\nonumber\eea By the \emph{Second Main Theorem}, we have \bea\label{e2.5} T(r,\Psi(f))&\leq& N\left(r,\frac{1}{\Psi(f)}\right) +N\left(\frac{1}{\Psi(f)-b}\right)+N\left(r,\frac{1}{\Psi(f)-c}\right)-N_1(r,\Psi(f))\nonumber \\&&+S(r,f),  \eea where $ N_1(r,\Psi(f))=2N(r,\Psi(f))-N(r,\Psi^{\prime}(f))+N\left(r,\displaystyle\frac{1}{\Psi^{\prime}(f)}\right). $\par Let $ z_0 $ be a pole of $ f $ with multiplicity $ m(\geq 1) $, then $ z_0 $ will be a pole of both $ \Psi(f) $ and $ \Psi^{\prime}(f) $,  of respective multiplicities $ 2mn+k+q $ and $ 2mn+k+q+1 $, where $ q=0 $, if $ z_0 $ is neither a pole nor a zero of $ \alpha $, $ q=t $, if $ z_0 $ is a pole of $ \alpha $ with multiplicity $ t $, and $ q=-t $, if $ z_0 $ is a zero of $ \alpha $ with multiplicity $ t $, where $ t $ is a positive integer.\par Thus we have \beas 2(2mn+k+q)-(2mn+k+q+1)&=& 2mn+k+q-1\\ &=& m+n+2mn+k+q-m-1\\&\geq& m+n, \eeas because $ 2mn+k+q-m-1\geq k-1\geq 0. $ \par Again since $ T(r,\alpha)=S(r,f) $, so it follows that \bea\label{e2.6}  N_1(r,\Psi(f))\geq N(r,\Psi(f))+N\left(r,\displaystyle\frac{1}{\Psi^{\prime}(f)}\right)+S(r,f).\eea\par We obtained from (\ref{e2.4}), (\ref{e2.5}) and (\ref{e2.6}), we get  \beas && T\left(r,\alpha [P(f)]^2 \right)\\&\leq& N\left(r,\frac{1}{\alpha [P(f)]^2}\right)+N\left(\frac{1}{\Psi(f)-b}\right)+N\left(r,\frac{1}{\Psi(f)-c}\right)-N(r,\Psi(f))\\&&-N\left(r,\displaystyle\frac{1}{\Psi^{\prime}(f)}\right)+S(r,f). \eeas i.e., \beas 2nT\left(r,f\right)&\leq& 2 N\left(r,\frac{1}{ P(f)}\right)+N\left(\frac{1}{\Psi(f)-b}\right)+N\left(r,\frac{1}{\Psi(f)-c}\right)\\&&-N(r,\Psi(f))-N\left(r,\displaystyle\frac{1}{\Psi^{\prime}(f)}\right)+S(r,f). \eeas
\end{proof}
\begin{lem}\label{lem2.8}
    Let $ \mathfrak{F} $ and $ \mathfrak{G} $ be two non-constant meromorphic functions defined as in (\ref{e2.1}) be such that $ \mathfrak{F}\equiv \mathfrak{G} $. If $ p>k+1 $ then
    \begin{enumerate}
        \item[(i).] $ Q(f_*) $ reduces to a non-zero monomial $ c_jf_{*}^j $ for some $ j\in \{0,1,2, \ldots,m\} $.
        \item[(ii).]  $ f(z) $ takes the form \beas f(z)=c e^{\frac{\lambda}{p+j}z}+d_p, \eeas where $ c $ is a non-zero constant and $ \lambda^k=1 $.
    \end{enumerate}
\end{lem}
\begin{proof}
    Since $ \mathfrak{F}\equiv \mathfrak{G} $ i.e., $ P(f)\equiv [P(f)]^{(k)} $ i.e., we have \bea \label{e2.7} f_{*}^pQ(f_*)=[f_{*}^pQ(f_*)]^{(k)} .\eea  (i). Our aim is to $ Q(f_8) $ reduces to a non-zero monomial $ c_jf_{*}^j $ for some $ j\in\{0, 1, 2, \ldots,m \}. $\par On contrary, let us suppose that $ Q(f_*)=c_mf_{*}^m+\ldots+c_1f^{*}+c_0 $, in which at least two terms present. \par it follows from (\ref{e2.7}) that $ f $ can not have any poles i.e., in other words $ f $ must be an entire function. \par Again since $ p>k+1 $, so one can check that $ 0 $ is an Picard exceptional value of $ f_* $. So, we have $ f_*=e^{h(z)}, $ where $ h $ is a non-constant entire function. Therefore, an elementary calculation shows that \bea\label{e2.8} c_j\bigg[f_{*}^{p+j}-\left(f_{*}^{p+j}\right)^{(k)}\bigg]=\phi_{j}\left(h^{\prime},\ldots, h^{(k)}\right)e^{(p+j)h},  \eea where $ \phi_{j}\equiv \phi_{j}\left(h^{\prime},\ldots, h^{(k)}\right) $ $ (j=0, 1, \ldots,m ) $ are differential polynomials in $ h^{\prime}, \ldots, h^{(k)}. $ \par From (\ref{e2.7}) and (\ref{e2.8}), we get that \bea\label{e2.9} \phi_{m}e^{mh}+\ldots+\phi_{1}e^{h}+\phi_{0}\equiv 0. \eea\par Since $ T(r,\phi_{j})=S(r,f) $ for $ (j=0, 1, \ldots, m), $ therefore by \emph{Borel unicity theorem} \cite[Theorem 1.52]{bibid}, one can get from (\ref{e2.9}) that $ \phi_{j}\equiv 0 $.\par Since $ Q(f_*) $ contains at least two terms, so there must exist $ s,\; t\in\{0, 1, 2,  \} $ with $ s\neq t $, so we must have from (\ref{e2.8}) that \beas f_{*}^{p+s}\equiv [f_{*}^{p+s}]^{(k)}\;\; \text{and}\;\; f_{*}^{p+t}\equiv [f_{*}^{p+t}]^{(k)},\eeas which is a contradiction, otherwise, in this case, the function $ f $ would have two different forms. \par Thus we see that $ Q(f_*)=c_jf_{*}^{j} $ for some $ j\in\{0, 1, 2, \ldots, m\} $.\par (ii). We note that the form of the function $ f_* $ satisfying $ f_{*}^{p+j}=[f_{*}^{p+j}]^{(k)} $ will be $ f_*=c e^{\frac{\lambda}{p+j}} $, where $ c $ is a non-zero constant and $ \lambda^k=1 $. Hence, we see that \beas f(z)=c e^{\frac{\lambda}{p+j}}+d_p.\eeas
\end{proof}
\section{Proof of Theorems}
\subsection{\sc Proof of Theorem \ref{th1.1}}
Let $ \mathfrak{F} $ and $ \mathfrak{G} $ be defined as in (\ref{e2.1}) and $ \mathfrak{H} $ be as in (\ref{e2.2}). Since $ P(f)-a $ and $ [P(f)]^{(k)}-a $ share $ (0,m) $, so it follows that $ \mathfrak{F} $ and $ \mathfrak{G} $ share $ (1,m) $.\par  Let $ \mathfrak{H}\not\equiv 0.$ So it follows from  Lemma \ref{lem2.3} and (i) of Lemma \ref{lem2.6}, we have \beas && T(r,\mathfrak{G})\\ &\leq& N_2(r,\mathfrak{F})+N_2\left(r,\frac{1}{\mathfrak{F}}\right)+N_2(r,\mathfrak{G})+N_2\left(r,\frac{1}{\mathfrak{G}}\right)+S(r,\mathfrak{F})+S(r,\mathfrak{G})\\&\leq& N_2(r,P(f))+N_2\left(r,\frac{1}{P(f)}\right)+N_2(r,[P(f)]^{(k)})+N_2\left(r,\frac{1}{[P(f)]^{(k)}}\right)\\&\leq& T\left(r,[P(f)]^{(k)}\right)-nT(r,f)+N_2\left(r,\frac{1}{P(f)}\right)+N_{k+2}\left(r,\frac{1}{P(f)}\right)\\&&+4\ol N(r,f)+S(r,f).  \eeas\par By \emph{Lemma \ref{lem2.1}}, we obtained \beas && T\left(r,[P(f)]^{(k)}\right) \\&\leq& T\left(r,[P(f)]^{(k)}\right)-nT(r,f)+N_2\left(r,\frac{1}{P(f)}\right)+N_{k+2}\left(r,\frac{1}{P(f)}\right)\\&&+4\ol N(r,f)+S(r,f).  \eeas i.e., \beas  nT(r,f) \leq 2N_{k+2}\left(r,\frac{1}{P(f)}\right)+4\ol N(r,f)+S(r,f).  \eeas This shows that, for any arbitrary $ \epsilon>0 $, \beas \bigg[2\delta_{k+2}\left(0,P(f)\right)+4\Theta(\infty,f)\bigg] T(r,f)\leq (6-n+\epsilon)T(r,f)+S(r,f), \eeas which contradicts \beas 2\delta_{k+2}\left(0,P(f)\right)+4\Theta(\infty,f)>6-n. \eeas \par Therefore, $ \mathfrak{H}\equiv 0 $. Integrating, twice, we get that \beas \frac{1}{\mathfrak{F}-1}=\frac{\mathfrak{A}}{\mathfrak{G}-1}+\mathfrak{B}, \eeas where $ \mathfrak{A}(\neq 0) $ and $ \mathfrak{B} $ are constants. \par Therefore, we have \bea\label{e3.1} \mathfrak{F}=\frac{(\mathfrak{B}+1)\mathfrak{G}+(\mathfrak{A}-\mathfrak{B}-1)}{\mathfrak{BG}+(\mathfrak{A}-\mathfrak{B})} \eea and hence we see that \beas T(r,\mathfrak{F})=T(r,\mathfrak{G})+S(r,f). \eeas
\par We are at a position to discuss the following three cases.\\
\noindent{\sc Case 1.} Suppose that $ \mathfrak{B}\neq -1, 0 $.\\
\noindent{\sc Subcase 1.1.}
If $ \mathfrak{A}-\mathfrak{B}-1\neq 0 $, then from (\ref{e3.1}), we have \beas \ol N\left(r,\frac{1}{\mathfrak{G}+\displaystyle\frac{\mathfrak{A}-\mathfrak{B}-1}{\mathfrak{B}+1}}\right)=\ol N\left(r,\frac{1}{\mathfrak{F}}\right). \eeas\par Applying \emph{Second Main Theorem}, we have \beas  T(r,\mathfrak{G})&\leq& \ol N(r,\mathfrak{G})+\ol N\left(r,\frac{1}{\mathfrak{G}}\right)+\ol N\left(r,\frac{1}{\mathfrak{G}+\displaystyle\frac{\mathfrak{A}-\mathfrak{B}-1}{\mathfrak{B}+1}}\right)+S(r,\mathfrak{G})\\&\leq& \ol N(r,f)+\ol N\left(r,\frac{1}{[P(f)]^{(k)}}\right)+\ol N\left(r,\frac{1}{P(f)}\right)+S(r,f)\\&\leq& \ol N(r,f)+ T\left(r,[P(f)]^{(k)}\right)-nT(r,f)+\ol N_{k+2}\left(r,\frac{1}{P(f)}\right)\\&&+\ol N\left(r,\frac{1}{P(f)}\right)+S(r,f). \eeas i.e., we have \beas && T\left(r,[P(f)]^{(k)}\right)\\&\leq& \ol N(r,f)+ T\left(r,[P(f)]^{(k)}\right)-nT(r,f)+\ol N_{k+2}\left(r,\frac{1}{P(f)}\right)+\ol N\left(r,\frac{1}{P(f)}\right)\\&&+S(r,f), \eeas and so we have \beas && nT(r,f) \\&\leq& \ol N(r,f)+\ol N_{k+2}\left(r,\frac{1}{P(f)}\right)+\ol N\left(r,\frac{1}{P(f)}\right)+S(r,f)\\&\leq& \ol N(r,f)+2\ol N_{k+2}\left(r,\frac{1}{P(f)}\right)+S(r,f).\eeas
This shows that, for any arbitrary $ \epsilon>0 $, \beas \bigg[2\delta_{k+2}\left(0,P(f)\right)+\Theta(\infty,f)\bigg] T(r,f)\leq (3-n+\epsilon)T(r,f)+S(r,f), \eeas which contradicts \beas 2\delta_{k+2}\left(0,P(f)\right)+4\Theta(\infty,f)>6-n. \eeas
\\
\noindent{\sc Subcase 1.2.} Thus we have $ \mathfrak{A}-\mathfrak{B}-1=0 $. Then it follows from (\ref{e3.1}) that \beas \ol N\left(r,\frac{1}{\mathfrak{G}+\displaystyle\frac{1}{\mathfrak{B}}}\right)=\ol N(r,\mathfrak{F}). \eeas \par By the same argument as above, we can reached in a contradiction.\\
\noindent{\sc Case 2.} Let $ \mathfrak{B}=-1 $.
\\
\noindent{\sc Subcase 2.1.} Suppose that $ \mathfrak{A}+1\neq 0 $.
Then we have from (\ref{e3.1}) \beas \ol N\left(r,\frac{1}{\mathfrak{G}_(\mathfrak{A}+1)}\right)=\ol N(r,\mathfrak{F}). \eeas
\par By the similar argument to the \emph{Case 1}, we can arrive at a contradiction.\\
\noindent{\sc Subcase 2.2.} Let $ \mathfrak{A}+1= 0 $, then from (\ref{e3.1}), we see that $ \mathfrak{FG}=1 $. i.e., we have \bea\label{e3.2} P(f)[P(f)]^{(k)}=a^2. \eea
\\
\noindent{\sc Subcase 2.2.1.} Let $ f $ be a rational function. Then, $ P(f) $, and hence $ [P(f)]^{(k)} $ will also be a rational function. Therefore, from (\ref{e3.2}), we see that $ a $ is a non-zero constant. So from (\ref{e2.1}), we see that $ P(f) $ has no zero and pole. Since $ f $ is non-constant, hence we arrive at a contradiction.
\\
\noindent{\sc Subcase 2.2.2.} Let $ f $ be a transcendental meromorphic function. Then by \emph{Lemma \ref{lem2.7}} in view of (\ref{e3.2}), we have \beas 2n T(r,f)&\leq& 2 N\left(r,\frac{1}{P(f)}\right)+2T\left(r,P(f)[P(f)]^{(k)}\right)+S(r,f)\\&\leq & 2 N\left(r,\frac{1}{P(f)}\right)+2T\left(r,a^2\right)+S(r,f)\\&\leq & 2 N\left(r,\frac{1}{a^2}\right)+S(r,f)\\&\leq& S(r,f),  \eeas which is a contradiction. \\
\noindent{\sc Case 3.} Suppose that $ \mathfrak{B}\equiv 0 $.
\\
\noindent{\sc Subcase 3.1.} Let $ \mathfrak{A}-1\neq 0 $, then we see from (\ref{e3.1}) that \beas \ol N\left(r,\frac{1}{\mathfrak{G}+(\mathfrak{A}-1)}\right)=\ol N\left(r,\frac{1}{\mathfrak{F}}\right). \eeas\par Similar to the argument as in \emph{Case 1}, we can arrive at a contradiction.
\noindent{\sc Subcase 3.2.} Therefore, we have $ \mathfrak{A}-1= 0 $. So we obtain  from (\ref{e3.1}) that $ \mathfrak{F}\equiv \mathfrak{G}.$ Now by Applying \emph{Lemma \ref{lem2.8}}, we see that $ Q(f_*) $ reduces to a non-zero monomial $ c_jf_{*}^p $, and hence, the function $ f $ takes the form  \beas f(z)=c e^{\frac{\lambda}{p+j}}+d_p, \eeas where $ c $ is a non-zero constant and $ \lambda^k=1. $
\subsection{\sc Proof of Theorem \ref{th1.2}}
Let $ \mathfrak{F} $ and $ \mathfrak{G} $ be defined by (\ref{e2.1}). Then it is clear that $ \mathfrak{F}-1=\displaystyle\frac{P(f)-a}{a} $ and $ \mathfrak{G}=\displaystyle\frac{[P(f)]^{(k)}-a}{a}. $\par Since $ P(f)-a $ and $ [P(f)]^{(k)}-a $ share $(0,1) $, so it follows that $ \mathfrak{F} $ and $ \mathfrak{G} $ share $ (1,1) $ except the zeros and poles of $ a(z). $
Let $ \mathfrak{H} $ be defined as in (\ref{e2.2}) and we suppose that $ \mathfrak{H}\not\equiv 0 $. Then by \emph{Lemma \ref{lem2.6}}, we have \bea&& \label{e33.3} T(r,\mathfrak{G})\\&\leq &\nonumber N_2(r,\mathfrak{F})+N_2\left(r,\frac{1}{\mathfrak{F}}\right)+N_2(r,\mathfrak{G})+N_2\left(r,\frac{1}{\mathfrak{G}}\right)+\ol N^{L}\left(r,\frac{1}{\mathfrak{G}-1}\right)\\&&+S(r,\mathfrak{F})+S(r,\mathfrak{G})\nonumber. \eea \par Next we see that \beas \ol N^{L}\left(r,\frac{1}{\mathfrak{G}-1}\right)&\leq& \frac{1}{2} N\left(r,\frac{\mathfrak{G}}{\mathfrak{G}^{\prime}}\right)\\&\leq& N\left(r,\frac{\mathfrak{G}^{\prime}}{\mathfrak{G}}\right)+S(r,\mathfrak{G})\\&\leq& \frac{1}{2} \ol N(r,\mathfrak{G})+\frac{1}{2}\ol N\left(r,\frac{1}{\mathfrak{G}^{\prime}}\right)+S(r,\mathfrak{G})\\&\leq& \frac{1}{2}\ol N(r,f)+\frac{1}{2}\ol N\left(r,\frac{1}{[P(f)]^{(k)}}\right)+S(r,f). \eeas\par In view of \emph{Lemma \ref{lem2.3}}, we get from (\ref{e33.3}) that \beas && T\left(r,[P(f)]^{(k)}\right)\\ &\leq& N_2(r,f)+N_2\left(r,\frac{1}{P(f)}\right)+N_2(r,[P(f)]^{(k)})+\frac{1}{2}\ol N(r,f)+\frac{1}{2}\ol N\left(r,\frac{1}{[P(f)]^{(k)}}\right)\\&&+S(r,f)\\&\leq& N_{k+2}\left(r,\frac{1}{P(f)}\right)+T\left(r,[P(f)]^{(k)}\right)-nT(r,f)+N_{k+2}\left(r,\frac{1}{P(f)}\right)\\&&+\frac{1}{2}N_{k+2}\left(r,\frac{1}{P(f)}\right)\frac{k+9}{2}\ol N(r,f)+S(r,f). \eeas i.e., for any arbitrary $ \epsilon>0 $, we get \beas  nT(r,f)&\leq& \frac{5}{2}N_{k+2}\left(r,\frac{1}{P(f)}\right)+\frac{k+9}{2}\ol N(r,f)+S(r,f)\\&\leq& \bigg\{\frac{5n+k+9}{2}-\frac{5n}{2}\delta_{k+2}(0,P(f))-\frac{k+9}{2}\Theta(\infty,f)+\epsilon\bigg\}T(r,f)\\&&+S(r,f), \eeas which shows that
\beas {5n}\delta_{k+2}(0,P(f))+(k+9)\Theta(\infty,f)\leq{3n+k+9}, \eeas which contradicts \beas {5n}\delta_{k+2}(0,P(f))+(k+9)\Theta(\infty,f)>{3n+k+9}. \eeas\par Thus we have $ \mathfrak{H}\equiv 0 $. Next  proceeding exactly same way as done in the proof \emph{Theorem \ref{th1.1}}, we get the conclusion of \emph{Theorem \ref{th1.2}}.

\subsection{\sc Proof of Theorem \ref{th1.3}}
Let $ \mathfrak{F} $ and $ \mathfrak{G} $ be defined by (\ref{e2.1}). Then it is clear that $ \mathfrak{F}-1=\displaystyle\frac{P(f)-a}{a} $ and $ \mathfrak{G}=\displaystyle\frac{[P(f)]^{(k)}-a}{a}. $\par Since $ P(f)-a $ and $ [P(f)]^{(k)}-a $ share $ 0 $ $ IM $, so it follows that $ \mathfrak{F} $ and $ \mathfrak{G} $ share $ 1 $ except the zeros and poles of $ a(z). $ \par We suppose that $ \mathfrak{F}\not\equiv\mathfrak{G}.$ Let \beas  \Psi &=& \frac{1}{\mathfrak{F}}\left(\frac{\mathfrak{G}^{\prime}}{\mathfrak{G}-1}-(k+1)\frac{\mathfrak{F}^{\prime}}{\mathfrak{F}-1}\right)\\ &=& \frac{\mathfrak{G}}{\mathfrak{F}}\left(\frac{\mathfrak{G}^{\prime}}{\mathfrak{G}-1}-\frac{\mathfrak{G}^{\prime}}{\mathfrak{G}}\right)-(k+1)\left(\frac{\mathfrak{F}^{\prime}}{\mathfrak{F}-1}-\frac{\mathfrak{F}^{\prime}}{\mathfrak{F}}\right). \eeas \par We first suppose that $ \Psi\equiv 0 $. Then, we have \bea\label{e3.3} \mathfrak{G}-1=c \left(\mathfrak{F}-1\right)^{(k+1)}, \eea where $ c $ is a non-zero complex constants.\par Let $ z_0 $ be a pole of $ f $ with multiplicity $ p(\geq 1) $ such that $ a(z_0)\neq 0, \infty $. Clearly, $ z_0 $ is a pole of $ \mathfrak{G}-1 $ with multiplicity $ np+k $, and a pole of $ (\mathfrak{G}-1)^{k+1} $ with multiplicity $ np(k+1). $ Then it follows from (\ref{e3.3}),  we must have  $ np+k=np(k+1). $ If $ np\geq 2 $, then we arrive at a contradiction, and so we have \bea\label{e3.4} N_{(2}(r,f)=S(r,f). \eea\par By our assumption is \beas (k-1)\Theta(\infty,f)+\delta_{2}(\infty,f)+n[\delta_k(0,P(f))+\delta_{k+1}(0,P(f))]>k+n \eeas, which in turn shows that \beas (k-1)\Theta(\infty,f)+\delta_{2}(\infty,f)+n\delta_k(0,P(f))>n \eeas By \emph{Lemma \ref{lem2.1}}, we have \beas n(k+1) T(r,f)&=& (k+1) T(r,\mathfrak{F})+S(r,f)\\&\leq& T\left(r,(\mathfrak{F}-1)^{k+1}\right)+S(r,f)\\&\leq& T(r,\mathfrak{G})+S(r,f)\\&\leq&  T\left(r,\frac{\mathfrak{G}}{\mathfrak{F}}\right)+T(r,\mathfrak{F})+S(r,f)\\&\leq&  N\left(r,\frac{[P(f)]^{(k)}}{P(f)}\right)+m\left(r,\frac{[P(f)]^{(k)}}{P(f)}\right)+S(r,f)\\&\leq& k \ol N(r,f)+N_k(r,0;P(f))+nT(r,f)+S(r,f)\\&\leq& (k-1)\ol N(r,f)+ N_2(r,f)+N_k(r,0;P(f))+nT(r,f)\\&&+S(r,f)\\&\leq& \left(k+2n- (k-1)\Theta(\infty,f)-\delta_{2}(\infty,f)-n\delta_k(0, P(f)+\epsilon)\right)T(r,f)\\&&+S(r,f), \eeas which implies $ (k-1)\Theta(\infty,f)+\delta_{2}(\infty,f)+n\delta_k(0,P(f))\leq n+(1-k)n\leq n, $ and this contradicts \beas (k-1)\Theta(\infty,f)+\delta_{2}(\infty,f)+n\delta_k(0,P(f))>n. \eeas \par Thus we see that $ \Psi\not\equiv 0 $.\par From the \emph{Fundamental estimate of logarithmic derivative} it follows that $ m(r, \Psi)=S(r,f). $ \par Let $ z_1 $ be a pole of $ f $ with multiplicity $ q(\geq 1) $, such that $ a(z_1)\neq 0, \infty $, then we see that $$
\Psi(z) = \left\{
\begin{array}{ll}
O((z-z_1)^n) & \quad q=1 \\
O((z-z_1)^{n(q-1)}) & \quad q\geq 2
\end{array}
\right.
$$
\par In view of the definition of $ \Psi $, we have \bea\label{e3.5} && N(r,\Psi)\\&\leq&\nonumber \ol N\left(r,\frac{1}{\mathfrak{F}}\right)+\ol N_{k+1}\left(r,\frac{1}{P(f)}\right)+S(r,f)\\&\leq& N\left(r,\frac{1}{\frac{\mathfrak{F}-\mathfrak{G}}{\mathfrak{F}}}\right)+\ol N_{k+1}\left(r,\frac{1}{P(f)}\right)+S(r,f)\nonumber\\&\leq& T\left(r,\frac{\mathfrak{G}}{\mathfrak{F}}\right)+\ol N_{k+1}\left(r,\frac{1}{P(f)}\right)+S(r,f)\nonumber\\&\leq& N\left(r,\frac{[P(f)]^{(k)}}{P(f)}\right)+m\left(r,\frac{[P(f)]^{(k)}}{P(f)}\right)+\ol N_{k+1}\left(r,\frac{1}{P(f)}\right)+S(r,f)\nonumber\\&\leq& k\ol N(r,f)+N_k\left(r,\frac{1}{P(f)}\right)+\ol N_{k+1}\left(r,\frac{1}{P(f)}\right)+S(r,f)\nonumber. \eea\par Using (\ref{e3.5}), we have \beas && N(r,f)-\ol N_{(2}(r,f)\\&\leq& N(r,0;\Psi)\\&\leq& T\left(r,\frac{1}{\Psi}\right)-m\left(r,\frac{1}{\Psi}\right)+S(r,f)\\&\leq& T\left(r,\Psi \right)-m\left(r,\frac{1}{\Psi}\right)+S(r,f)\\&\leq& N(r,\Psi)+m(r,\Psi)-m\left(r,\frac{1}{\Psi}\right)+S(r,f)\\&\leq& k \ol N(r,f)+N_{k}\left(r,\frac{1}{P(f)}\right)+N_{k+1}\left(r,\frac{1}{P(f)}\right)-m\left(r,\frac{1}{\Psi}\right)+S(r,f). \eeas \par Again from the definition of $ \Psi $, we see that \bea\label{e3.6} m(r,P(f))\leq m\left(r,\frac{1}{\Psi}\right)+S(r,f). \eea \par Now it follows from (\ref{e3.5}) and (\ref{e3.6}), that \beas && nT(r,f)\\&\leq& (k-1) \ol N(r,f)+N_2(r,f)+N_{k}\left(r,\frac{1}{P(f)}\right)+N_{k+1}\left(r,\frac{1}{P(f)}\right)+S(r,f)\\&\leq& \nonumber\bigg\{k+2n (k-1)\Theta(\infty,f)+\delta_{2}(\infty,f)+n[\delta_k(0,P(f))+\delta_{k+1}(0,P(f))]+\epsilon\bigg\}T(r,f)\\&&+S(r,f),  \eeas
which contradicts \beas (k-1)\Theta(\infty,f)+\delta_{2}(\infty,f)+n[\delta_k(0,P(f))+\delta_{k+1}(0,P(f))]>k+n. \eeas\par Therefore, we must have $ \mathfrak{F}\equiv \mathfrak{G} $ i.e., $ P(f)\equiv [P(f)]^{(k)} $. By \emph{Lemma  \ref{lem2.8}}, if $ p>k+1 $ then
\begin{enumerate}
    \item[(i).] $ Q(f_*) $ reduces to a non-zero monomial $ c_jf_{*}^j $ for some $ j\in \{0,1,2, \ldots,m\} $.
    \item[(ii).]  $ f(z) $ takes the form \beas f(z)=c e^{\frac{\lambda}{p+j}z}+d_p, \eeas where $ c $ is a non-zero constant and $ \lambda^k=1 $.
\end{enumerate}


\end{document}